\documentclass[12pt]{article}
\usepackage[top=1in, bottom=1in, inner=1in, outer=1in, footskip=0.5in]{geometry}
\usepackage{amsmath}
\usepackage{amsthm}
\usepackage{amssymb}
\usepackage{enumerate}
\usepackage{color}
\usepackage{tikz}
\usepackage{multirow}
\usepackage{cancel}
\allowdisplaybreaks

\newtheorem{theorem}{Theorem}[section]
\newtheorem{lemma}[theorem]{Lemma}
\newtheorem{proposition}[theorem]{Proposition}
\newtheorem{corollary}[theorem]{Corollary}

\theoremstyle{definition}
\newtheorem{definition}[theorem]{Definition}

\newtheorem{remark}[theorem]{Remark}

\begin{document}

\baselineskip=17pt

\title{A coding of bundle graphs and their embeddings into Banach spaces}
\author{Andrew Swift \\
E-mail:  ats0@math.tamu.edu  }
\date{}
\maketitle

\renewcommand{\thefootnote}{}


\begin{abstract}
The purpose of this article is to generalize some known characterizations of Banach space properties in terms of graph preclusion.  In particular, it is shown that superreflexivity can be characterized by the non-equi-bi-Lipschitz embeddability of any family of bundle graphs generated by a nontrivial finitely-branching bundle graph.  It is likewise shown that asymptotic uniform convexifiability can be characterized within the class of reflexive Banach spaces with an unconditional asymptotic structure by the non-equi-bi-Lipschitz embeddability of any family of bundle graphs generated by a nontrivial $\aleph_0$-branching bundle graph.  The best known distortions are recovered.  For the specific case of $L_1$, it is shown that every countably-branching bundle graph bi-Lipschitzly embeds into $L_1$ with distortion no worse than $2$.
\end{abstract}

\section{Introduction}

Recall that given two metric spaces $(X,d_X)$ and $(Y,d_Y)$, $X$ is said to be bi-Lipschitzly embeddable into $Y$ if there is a function $f\colon X\to Y$ and constants $C_1,C_2>0$ such that 
\begin{equation}
\label{eq:lipschitz}
C_1d_X(x_1,x_2)\leq d_Y(f(x_1),f(x_2))\leq C_2d_X(x_1,x_2)
\end{equation}
for all $x_1,x_2\in X$, and in this case $f$ is called a bi-Lipschitz embedding.  The distortion $\mathrm{dist}(f)$ of a bi-Lipschitz embedding $f$ is the infimum of $C_2/C_1$ over all constants $C_1,C_2>0$ satisfying \eqref{eq:lipschitz}.  We'll let $c_Y(X)$ be the infimum of $\mathrm{dist}(f)$ over all bi-Lipschitz embeddings $f\colon X\to Y$.  A family of metric spaces $\{X_i\}_{i\in \mathcal{I}}$ is said to be equi-bi-Lipschitzly embeddable into $Y$ if $\sup_{i\in \mathcal{I}} c_Y(X_i)<\infty$.

In \cite{bourgain}, J. Bourgain proved that the notion of superreflexivity in Banach spaces can be characterized by the non-equi-bi-Lipschitz embeddability of the family of binary trees with finite height.  Since then, the non-equi-bi-Lipschitz embeddability of several other families of graphs have also been shown to characterize superreflexivity (\cite{baudier}, \cite{johnson_schechtman}, \cite{ostrovskii_randrianantoanina}).  In \cite{baudier_etal}, F. Baudier et al. proved that the non-equi-bi-Lipschitz embeddability of the family of $\aleph_0$-branching diamond graphs characterizes the asymptotic uniform convexifiability of refexive Banach spaces with an unconditional asymptotic structure.  They also show that this same family of graphs is equi-bi-Lipschitzly embeddable into $L_1$.

The families of graphs used in \cite{johnson_schechtman}, \cite{ostrovskii_randrianantoanina}, and \cite{baudier_etal} are all contained in a larger class of graphs, called the ``bundle graphs''.  The goal of this paper is to generalize the results mentioned above to this larger class while providing simpler proofs.  We order the sections roughly in terms of ease of proof.  In Section \ref{sec:notation}, we'll define what a bundle graph is and provide a natural labelling of the vertices of such graphs.  We'll then derive a formula for the graph metric in terms of this labelling.  In Sections \ref{sec:auc} and \ref{sec:L1}, we'll generalize two results in \cite{baudier_etal}.  In Section \ref{sec:auc} we'll show that every countably-branching bundle graph is bi-Lipschitzly embeddable into any Banach space with a good $\ell_\infty$-tree with distortion bounded above by a constant depending only on the good $\ell_\infty$-tree, which implies a more general characterization of asymptotic uniform convexifiability for the class of reflexive Banach spaces with an unconditional asymptotic structure. In Section \ref{sec:L1} we'll show that every countably-branching bundle graph is bi-Lipschitzly embeddable into $L_1$ with distortion bounded above by $2$.  In Section \ref{sec:superreflexivity} we'll show that every finitely branching bundle graph is bi-Lipschitzly embeddable into any Banach space containing an equal-signs-additive basis with distortion bounded above by a constant not depending on the branching number (although it will still depend on the bundle graph).  However, in Section \ref{sec:oslash}, we'll show that this constant does not increase with $\oslash$-products, and thus generalize the characterizations of superreflexivity found in \cite{johnson_schechtman} and \cite{ostrovskii_randrianantoanina}.  

The problem of characterizing superreflexivity in purely metric terms belongs to a more general investigation of metric characterizations of local properties of Banach spaces, called the Ribe program.  A survey of other results in this program can be found in \cite{naor}.

\section{Notation and definitions}
\label{sec:notation}
We will denote $\mathbb{N}\cup \{0\}$ by $\mathbb{N}_0$ and given $n\in \mathbb{N}_0$, we will denote the set $\{i\in \mathbb{N}_0\ |\ i\leq n\}$ by $[n]$.
Given a finite sequence $A=(a_i)_{i=1}^n$, the length of $A$, denoted by $|A|$, is defined to be $n$ and the maximum of $A$, denoted by $\max A$ is defined to be $\max\{a_i\}_{i=1}^n$.  If $m\in \mathbb{N}_0$, then we'll define $A\restriction_m$ by 
$A\restriction_m=(a_i)_{i=1}^m$ if $m\leq n$ and $A\restriction_m=A$ if $m>n$.
We'll write $B\preceq A$ if $B=A\restriction_m$ for some $m\in \mathbb{N}_0$.  Given another finite sequence $B$, we'll denote by $A\wedge B$ the longest sequence $C$ such that $C\preceq A$ and $C\preceq B$, and by $A^\frown B$ the concatenation of $A$ and $B$.  Note that if $A_1\preceq A_2$ and $A_1\npreceq B$, then $A_2\wedge B=A_1\wedge B$.  We'll denote the sequence of length $0$ (the empty sequence) by $\emptyset$.  Given a set $X$ and $n\in \mathbb{N}_0$, we'll denote by $X^n$ the set of sequences in $X$ with length equal to $n$ and by $X^{\leq n}$ the set of all sequences in $X$ with length at most $n$.

Given a graph $G$, we'll always use the shortest-path metric when discussing the distance between two vertices in $G$.  We'll denote the vertex set of $G$ by $V(G)$ and the edge set of $G$ by $E(G)$.
A graph with two distinguished vertices, one designated the ``top'', and the other the ``bottom'', will be called a top-bottom graph.  The height of a top-bottom graph is defined to be the distance between its top and bottom.

\begin{definition}
\label{bundle_graph}
Given a cardinality $\kappa$, a top-bottom graph is called a \emph{$\kappa$-branching bundle graph} if it can be formed by any (finite) sequence of the following operations:
\begin{itemize}
\item (Initialization) Create a path of length 1, with one endpoint designated the top and the other the bottom.
\item (Parallel Composition) Given two $\kappa$-branching bundle graphs $G_1$ and $G_2$, create a new graph $G$ by identifying the top of $G_1$ with the bottom of $G_2$.  The bottom of $G$ will be the bottom of $G_1$ and the top of $G$ will be the top of $G_2$.
\item (Series Composition) Given a $\kappa$-branching bundle graph $G'$, create a new graph $G$ by taking $\kappa$ copies of $G'$ and then identifying all the bottoms with each other and all the tops with each other.  The bottom of $G$ will be the bottom of $G'$ and the top of $G$ will be the top of $G'$. 
\end{itemize} 
\end{definition}

\begin{definition}
\label{height_depth}
Given a vertex $v$ in a bundle graph $G$, the \emph{height} of $v$ is the distance from $v$ to the bottom of $G$.  The \emph{depth} or \emph{level} of $v$ is defined recursively as follows:
\begin{itemize}
\item If $v$ is the top or bottom of $G$, then $v$ has depth 0.
\item If $v$ is neither the top nor bottom of $G$, and $G$ was constructed via parallel composition between two $\kappa$-branching bundle graphs $G_1$ and $G_2$, then the depth of $v$ in $G$ is the same as its depth in $G_1$ if $v\in V(G_1)$ or its depth in $G_2$ if $v\in V(G_2)$.
\item If $v$ is neither the top nor bottom of $G$, and $G$ was constructed via series composition of a $\kappa$-branching bundle graph $G'$, then the depth of $v$ in $G$ is one more than the depth of $v'$ in $G'$ if $v$ is a copy of $v'\in V(G')$. 
\end{itemize}
\end{definition}

It isn't difficult to see that given two $\kappa$-branching bundle graphs, $G$ and $G'$, a new $\kappa$-branching bundle graph can be created by replacing every edge of $G$ with a copy of $G'$ (where the bottom of $G'$ is placed on the lower endpoint of the edge and the top on the higher).  We'll give a proof of this fact in Section \ref{sec:oslash}.  Thus the diamond and Laakso graphs used in \cite{johnson_schechtman}, \cite{ostrovskii_randrianantoanina}, and \cite{baudier_etal} are all examples of bundle graphs.

Suppose $G$ is a bundle graph with height $M+1$ for some $M\in \mathbb{N}_0$.  From the definitions, every vertex of $G$ at a given height will have the same depth.  And if we know the depth associated with each height, we can use Definition \ref{height_depth} to go backwards to find a sequence of operations from Definition \ref{bundle_graph} that can be used to create $G$.  Thus, when deriving properties of $G$, we don't actually need to know the sequence of operations used to create $G$.  All information about $G$ is contained in the sequence $W=(w_r)_{r=0}^{M+1}$, where $w_r$ is the depth associated to height $r$ (we include $w_0=w_{M+1}=0$ for convenience).

Suppose $r\in [M+1]$ is such that $w_r>0$, and let $v$ be a vertex of $G$ with height $r$.  Since $w_r>0$, $v$ is a copy of some vertex $v'$ in some $\kappa$-branching bundle graph $G'$ (which was used in series composition in one of the steps to create $G$).  To distinguish $v$ from other copies of $v'$, we'll label the copies of $G'$ with elements of $\kappa$ and record the copy in which $v$ was found as $a_{w_r}$.  Now in the graph $G'$, $v'$ has depth $w_r-1$.  If $w_r-1>0$, then we repeat this process for $v'$ and obtain $a_{w_r-1}$.  We continue to repeat this process until we obtain a sequence $A=(a_i)_{i=1}^{w_r}$ which can be used to distinguish $v$ from any other vertex at height $r$.  Doing this for every vertex in $G$ yields a labelling of the vertex set.

With this idea in mind, we are now in the position to give a non-recursive definition of bundle graph equivalent to Definition \ref{bundle_graph}.  Note that two adjacent vertices $u$ and $v$ of a bundle graph must differ in height by exactly 1.  Furthermore, if $u$ and $v$ are adjacent and were created during series composition of a bundle graph $G'$, then $u$ and $v$ must have come from the same copy of $G'$.

\begin{definition}
\label{def:alt_bundle_graph}
Given a finite sequence $W=(w_r)_{r=0}^{M+1}\subseteq \mathbb{N}_0$ such that $w_0=w_{M+1}=0$ and a cardinality $\kappa$, the \emph{$\kappa$-branching bundle graph associated with $W$} is $T_{W,\kappa}=(V,E)$, defined by
\begin{align*}
V&=\left\{(r,A)\ |\ r\in [M+1] \mbox{ and } A\in \kappa^{w_r}\right\}, \\
E&=\{\{(r,A),(s,B)\}\subseteq V\ |\ |r-s|=1 \mbox{ and } A\preceq B\}.
\end{align*}
The vertices $(0,\emptyset)$ and $(M+1,\emptyset)$ in $V$ are called the \emph{bottom} and \emph{top}, respectively, of $T_{W,\kappa}$.
\end{definition}

We illustrate in Figure \ref{example} below a typical bundle graph with its vertex labelling.

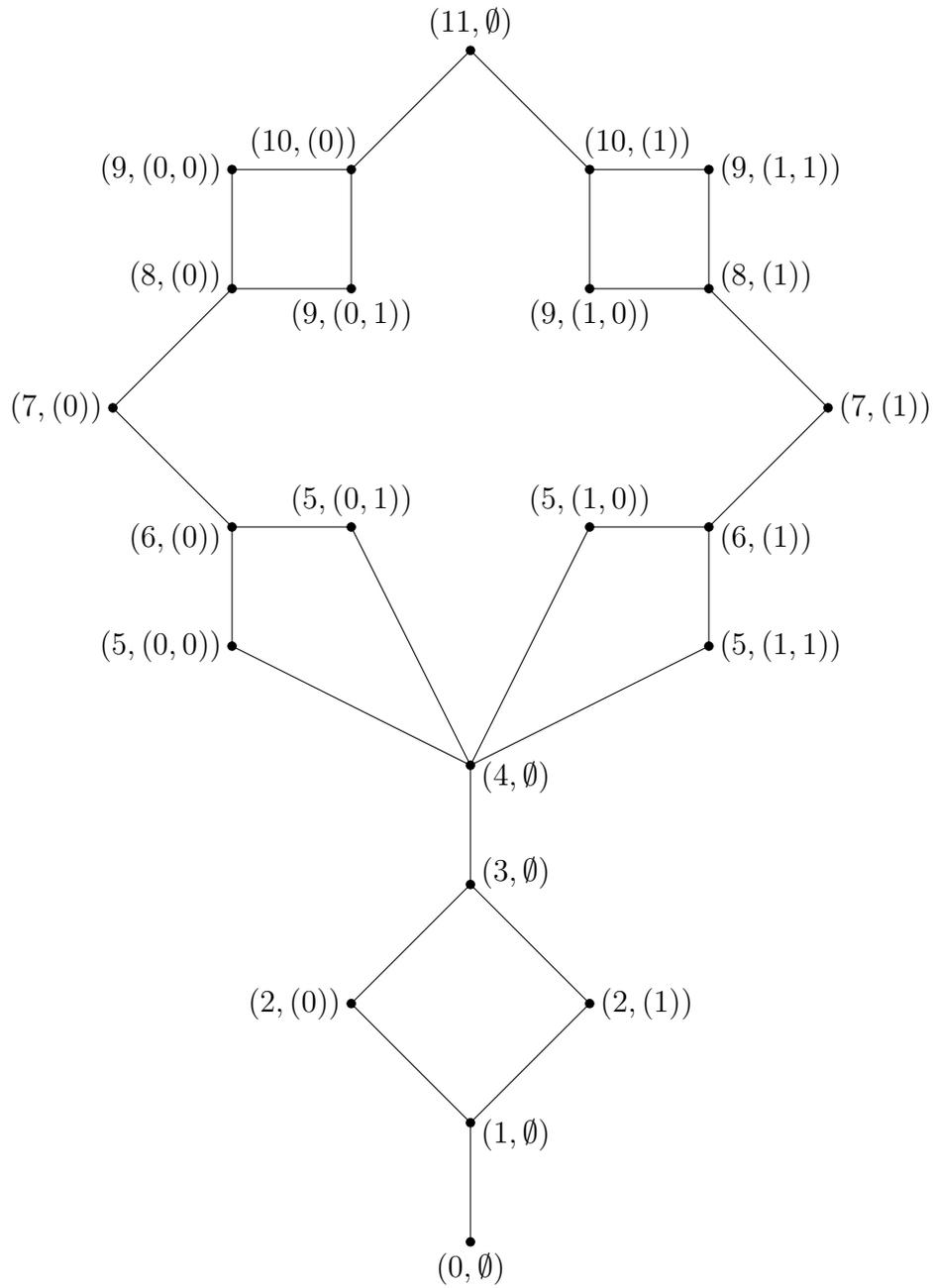
\begin{figure}

\centering
\begin{tikzpicture}[scale=0.8]
\filldraw 
(0,0) circle (2 pt) node[below] {$(0,\emptyset)$}

(0,2) circle (2 pt) 
(0,1.8) node[right]{$(1,\emptyset)$}
(-2,4) circle (2 pt) node[left] {$(2,(0))$}
(2,4) circle (2 pt) node[right] {$(2,(1))$}
(0,6) circle (2 pt) 
(0,6.2) node[right] {$(3,\emptyset)$}

(0,8) circle (2 pt) 
(0,7.8) node[right] {$(4,\emptyset)$}

(-2,12) circle (2 pt) node[above] {$(5,(0,1))$}
(-4,10) circle (2 pt) node[left] {$(5,(0,0))$} 
(-4,12) circle (2 pt) 
(-4,11.8) node[left] {$(6,(0))$}

(-6,14) circle (2 pt) node[left] {$(7,(0))$}

(2,12) circle (2 pt) node[above] {$(5,(1,0))$}
(4,10) circle (2 pt) node[right] {$(5,(1,1))$}
(4,12) circle (2 pt) 
(4,11.8) node[right] {$(6,(1))$}

(6,14) circle (2 pt) node[right] {$(7,(1))$}

(-4,16) circle (2 pt) 
(-4,16.2) node[left] {$(8,(0))$}
(-4,18) circle (2 pt) node[left] {$(9,(0,0))$}
(-2,16) circle (2 pt) node[below] {$(9,(0,1))$}
(-2,18) circle (2 pt) 
(-2.8,18) node[above] {$(10,(0))$}

(4,16) circle (2 pt) 
(4,16.2) node[right] {$(8,(1))$}
(4,18) circle (2 pt) node[right] {$(9,(1,1))$}
(2,16) circle (2 pt) node[below] {$(9,(1,0))$}
(2,18) circle (2 pt) 
(2.8,18) node[above] {$(10,(1))$}

(0,20) circle (2 pt) node[above] {$(11,\emptyset)$}

;

\draw
(0,0)

--(0,2) 
--(-2,4) 
--(0,6)
--(2,4) 
--(0,2)

(0,6)
--(0,8) 

--(-2,12) 
--(-4,12)
--(-4,10) 
--(0,8) 

(-4,12)
--(-6,14)

(0,8)
--(2,12) 
--(4,12)
--(4,10) 
--(0,8)

(4,12)
--(6,14)

(-6,14) 

--(-4,16) 
--(-4,18) 
--(-2,18)
--(-2,16) 
--(-4,16)

(-2,18)
--(0,20)

(6,14)
--(4,16) 
--(4,18) 
--(2,18)
--(2,16) 
--(4,16)

(2,18)
--(0,20) 
;

\end{tikzpicture}

\caption{$T_{W,\kappa}$ with $W=(0,0,1,0,0,2,1,1,1,2,1,0)$ and $\kappa=2$.} 
\label{example}
\end{figure}

\begin{remark}
If we don't specify the branching cardinality $\kappa$, then many bundle graphs have multiple representations from Definition \ref{def:alt_bundle_graph}.  For instance, $T_{(0,2,0),2}$ and $T_{(0,1,0,),4}$ are graph isomorphic (both represent a diamond graph of height 2 with 4 midpoints between top and bottom).  In the first case we think of the graph as being a 2-branching graph and in the second a 4-branching graph.  If we wanted to represent our bundle graphs uniquely, we could combine Definitions \ref{bundle_graph} and \ref{height_depth} into one definition and then require series composition to be allowable only if the bundle graph being copied has a vertex with depth 0 which is neither top nor bottom.  This would induce other requirements on $W$ besides $w_0=w_{M+1}=0$ and yield a unique naming of all bundle graphs (in this case $T_{(0,1,0,),4}$ would be the canonical representation for our example).  However, this is an unnecessary complication for the purpose of this paper.
\end{remark}

\begin{remark}
The only graphs this paper deals with are bundle graphs.  However, in some cases, results concerning other graphs can be recovered.  Note for instance that every tree is (isometrically) contained in some bundle graph as a subgraph.  Indeed, a $\kappa$-branching tree of finite height can be ``doubled'' to obtain a bundle graph containing the tree as its lower half.  For instance, the binary tree of height $3$ is contained in $T_{(0,1,2,3,2,1,0),2}$.
\end{remark}

Now that we have a way to represent a bundle graph with Definition \ref{def:alt_bundle_graph}, the next order of business is deriving a formula for the shortest-path metric.

\begin{lemma}
\label{path}
Let $T_{W,\kappa}=(V,E)$ be a bundle graph and fix $u=(r,A)$ and $v=(s,B)$ in $V$.  If $(t_i,C_i)_{i=0}^n$ is a path between $u$ and $v$, then there is $i\in [n]$ such that $C_i\preceq A\wedge B$.
\end{lemma}

\begin{proof}
Suppose the statement is false and let $(t_i,C_i)_{i=0}^n$ be a path starting at $u$ and ending at $v$ such that $C_i\npreceq A\wedge B$ for all $i\in [n]$.  Then in particular, $C_0=A\npreceq B$ and so $|C_0|\geq |A\wedge B|+1$ (while $|C_0\wedge B|=|A\wedge B|$).  Now suppose $k\in [n-1]$ is such that for all $i\in [k]$, $|C_i|\geq |A\wedge B|+1$ while $|C_i\wedge B|=|A\wedge B|$ (which implies $A\wedge B \preceq C_k$).  Either $C_k\preceq C_{k+1}$ or $C_{k+1}\preceq C_k$.  If $C_k\preceq C_{k+1}$, then $|C_{k+1}|\geq |C_k|\geq |A\wedge B|+1$ while $|C_{k+1}\wedge B|=|C_k\wedge B|=|A\wedge B|$.  So suppose $C_{k+1}\preceq C_k$.  If $|C_{k+1}|\leq |A\wedge B|$, then $C_{k+1}\preceq A\wedge B\preceq B$ since $|C_k\wedge B|=|A\wedge B|$, contradicting the choice of path.  Thus $C_{k+1}\geq |A\wedge B|+1$, and so if $|C_{k+1}\wedge B|<|A\wedge B|$, then $C_{k+1}\npreceq B$, which implies $|A\wedge B|=|C_k\wedge B|=|C_{k+1}\wedge B|<|A\wedge B|$, a contradiction.  Thus $|C_{k+1}\wedge B|\geq |A\wedge B|$.  And if $|C_{k+1}\wedge B|>|A\wedge B|$, then $|A\wedge B|=|C_{k}\wedge B|\geq |C_{k+1}\wedge B|>|A\wedge B|$, a contradiction.  Thus $|C_{k+1}\wedge B|=|A\wedge B|$.  By induction, it has been shown that for every $k\in [n]$, $|C_k|\geq |A\wedge B|+1$ while $|C_k\wedge B|=|A\wedge B|$.
In particular, $|B|=|C_n|\geq |A\wedge B|+1>|A\wedge B|=|C_n\wedge B|=|B|$, a contradiction.  Thus no such path $(t_i,C_i)_{i=0}^n$ exists which falsifies the statement of the lemma, and so the lemma is proved.
\end{proof}

By Lemma \ref{path}, a path between two vertices $u=(r,A)$ and $v=(s,B)$ in a bundle graph must contain a vertex $(t,C)$ such that both $C\preceq A$ and $C\preceq B$.  There are two cases to consider when trying to create a shortest path between $u$ and $v$:  Either such a vertex $(t,C)$ can be found so that $t$ is between $r$ and $s$, or not.  In either case we have $w_t=|C|\leq |A\wedge B|$.

We introduce some notation to differentiate these two possibilities.
Given two vertices $u=(r,A)$ and $v=(s,B)$ in a bundle graph, we will write $u\Updownarrow v$ to mean that there is $t\in[M+1]$ between $r$ and $s$ (inclusive) such that $w_t\leq |A\wedge B|$.  We will write $u\ \cancel{\Updownarrow}\ v$ to mean the opposite.  Note, in particular, that $u\Updownarrow v$ if $A\preceq B$ and $|A\wedge B|<\min\{|A|,|B|\}$ if $u\ \cancel{\Updownarrow}\ v$.

\begin{definition}
\label{ancestor}
Given two vertices $u=(r,A)$ and $v=(s,B)$ in a bundle graph, $u$ is said to be an \emph{ancestor} of $v$ and $v$ is said to be a \emph{descendant} of $u$ if $u\Updownarrow v$ and $r\leq s$.
\end{definition}

We'll show in the following proposition that the distance between a vertex in a bundle graph and one of its ancestors or descendants is simply the difference in height between the two.  By Lemma \ref{path}, finding the distance between two arbitrary vertices in a bundle graph thus amounts to finding a highest common ancestor or lowest common descendant of the two vertices.  

For example, using Figure \ref{example}, one may see that $(5,(1,1))-(6,(1))-(7,(1))-(8,(1))-(9,(1,0))$ is a shortest path between $(5,(1,1))$ and $(9,(1,0))$.  But a path between $(5,(1,1))$ and $(9,(0,1))$ must first go through either $(4,\emptyset)$ or $(11,\emptyset)$.

Given two vertices $u=(r,A)$ and $v=(s,B)$ in a bundle graph $T_{W,\kappa}$, we define $n(u,v)$ and $m(u,v)$ by 
\begin{align*}
n(u,v)&=\max\{t\in [M+1]\ |\ w_t\leq |A\wedge B|\mbox{ and } t\leq \min\{r,s\}\},\\
m(u,v)&=
\min\{t\in[M+1]\ |\ w_t\leq |A\wedge B|\mbox{ and } t\geq \max\{r,s\}\}.
\end{align*}

Following the definitions, one sees that if $u\ \cancel{\Updownarrow}\ v$, then $\left(n(u,v), A\wedge B \restriction_{w_{n(u,v)}}\right)$ is the highest common ancestor of $u$ and $v$ and $\left(m(u,v), A\wedge B \restriction_{w_{m(u,v)}}\right)$ is the lowest common descendant of $u$ and $v$.

\begin{proposition}
\label{dist}
Let $T_{W,\kappa}=(V,E)$ be a bundle graph with shortest-path metric $d$, and fix $u=(r,A)$ and $v=(s,B)$ in $V$.  Then
\[d(u,v)=
\begin{cases}
|r-s| & u\Updownarrow v\\
\min\{r+s-2n(u,v), 2m(u,v)-(r+s)\} & u\ \cancel{\Updownarrow}\ v 
\end{cases}
\]
\end{proposition}

\begin{proof}
By the definition of the edge set $E$, $d(u,v)\geq |r-s|$.
Suppose first that $A\preceq B$.
Recursively construct the sequence $(C_i)_{i=0}^{|r-s|}$ by letting 
\[C_0=\begin{cases}
A & r\leq s\\
B & s<r
\end{cases}
\]
and given $C_{i-1}$ for $i\in [|r-s|]\setminus \{0\}$, choosing $C_i\in \kappa^{w_{\min \{r,s\}+i}}$ so that $C_{i-1}\wedge C_i\in \{C_{i-1},C_i\}$ and $C_i\wedge B\in\{C_i,B\}$.
Then $((\min \{r,s\}+i, C_i))_{i=0}^{|r-s|}$ is a path between $u$ and $v$, and so $d(u,v)\leq|r-s|$.  That is, $d(u,v)=|r-s|$.

Now, if $u\Updownarrow v$, then there is $t$ between $r$ and $s$ such that $|A\wedge B|\geq w_t$.  Thus, by what was shown above,
\begin{align*}
d(u,v)&\leq d\left(u, \left(t,A\wedge B\restriction_{w_t}\right)\right)+d\left(\left(t,A\wedge B\restriction_{w_t}\right),v\right)\\
&=|r-t|+|t-s|\\
&=|r-s| 
\end{align*} 
and so $d(u,v)=|r-s|$.

If $u\ \cancel{\Updownarrow}\ v$, then by Lemma \ref{path} a shortest path between $u$ and $v$ must contain a vertex that is either a common ancestor or a common descendant of $u$ and $v$.  The result follows from what was shown above by taking the minimum of lengths of paths between $u$ and $v$ that contain either the highest common ancestor or lowest common descendant.
\end{proof}

Continuing our example from Figure \ref{example}, one may check that $w_7=1\leq |(1)|=|(1,1)\wedge (1,0)|$.  Thus $(5,(1,1))\Updownarrow (9, (1,0))$ and so $d((5,(1,1),(9,(1,0)))=9-5=4$ by Proposition \ref{dist}.
Similarly, there is no $t\in[11]$ between $5$ and $9$ such that $w_t\leq 0=|\emptyset|=|(1,1)\wedge (0,1)|$, meaning $(5,(1,1))\ \cancel{\Updownarrow}\  (9,(1,0))$.  One may determine that $n((5,(1,1)),(9,(0,1)))=4$ and $m((5,(1,1)),(9,(0,1)))=11$.  By Proposition \ref{dist}, $d((5,(1,1)),(9,(0,1)))=\min\{5+9-2\cdot 4, 2\cdot 11 -(9+5)\}=6$.

Throughout the next few sections, we'll describe various bi-Lipschitz embeddings of bundle graphs into Banach spaces.  We define now the coefficients that will regularly be used, and fix for the rest of the paper $W=(w_r)_{r=0}^{M+1}\subseteq \mathbb{N}_0$ such that $w_0=w_{M+1}=0$.  For $r\in [M+1]$ and $i\in \mathbb{N}_0$, define $x(r,i)$, $y(r,i)$, and $z(r,i)$ by
\begin{align*}
x(r,i)&=\begin{cases}
0 & i=0\\
\max\{t\in [M+1]\ |\ w_t<i \mbox{ and } t\leq r\} & i>0,
\end{cases}\\
y(r,i)&=
\begin{cases}
M+1 & i=0\\
\min\{t\in [M+1]\ |\ w_t<i \mbox{ and } t\geq r\} & i>0,
\end{cases}\\
z(r,i) & =\begin{cases}
r & i=0\\
\min \{r-x(r,i), y(r,i)-r\} & i>0.
\end{cases}
\end{align*}

Given $r\in [M+1]$ and $i\in \mathbb{N}$, $x(r,i)$ records the last height no greater than $r$ in which the vertices of a bundle graph associated with $W$ have depth less than $i$.  Similarly, $y(r,i)$ records the first height no lesser than $r$ in which the vertices of a bundle graph associated with $W$ have depth less than $i$.  And $z(r,i)$ simply records the distance one would have to travel from height $r$ to get to a vertex with depth less than $i$ in a bundle graph associated with $W$.

\section{Embedding into Banach spaces with good $\ell_\infty$-trees}
\label{sec:auc}
In this section, we show that for any countable cardinality $\kappa$, $T_{W,\kappa}$ is bi-Lipschitzly embeddable into any Banach space with a good $\ell_\infty$-tree of height $\max W$ with distortion bounded above by a constant depending only on the good $\ell_\infty$-tree.  For each $n\in \mathbb{N}$, we define the function $T_n\colon \aleph_0^{\leq n}\to \aleph_0$ by $T_n(\sigma)= \sum_{r=1}^{|\sigma|}(\sigma(r)+1)$ for each $\sigma\in \aleph_0^{\leq n}$.

\begin{definition}
Fix $n\in \mathbb{N}$ and let $(\sigma_i)_{i=0}^\infty$ be an enumeration of $\aleph_0^{\leq n}$ such that $i_1\leq i_2$ whenever $T_n(\sigma_{i_1})\leq T_n(\sigma_{i_2})$.
Given a Banach space $X$ and $C,D>0$, a sequence $(y_{\sigma_i})_{i=0}^\infty\subseteq S_X$ is called a \emph{$(C,D)$-good $\ell_\infty$-tree of height $n$} if, given any $(\alpha_i)_{i=0}^\infty\subseteq \mathbb{R}$,
\begin{enumerate}[(i)]
\item $1/C\|(\alpha_i)_{i=0}^n\|_\infty\leq \|\sum_{B\preceq A}\alpha_{|B|}y_B\|_X\leq C\|(\alpha_i)_{i=0}^n\|_\infty$
for all $A\in \aleph_0^n$,
\item $\left\|\sum_{i=0}^{m_1} \alpha_i y_{\sigma_i}\right\|_X\leq D \left\|\sum_{i=0}^{m_2}\alpha_i y_{\sigma_i}\right\|_X$
for all $m_1,m_2\in \mathbb{N}_0$ such that $m_1\leq m_2$.
\end{enumerate}
\end{definition} 

The first condition states that every ``branch'' $(y_B)_{B\preceq A}$ of the good $\ell_\infty$-tree is $C^2$-equivalent to the unit vector basis of $\ell_\infty^{n+1}$.  The second condition states that the sequence making up the good $\ell_\infty$-tree is basic with basis constant less than or equal to $D$.

\begin{theorem}
\label{thm:ell_inf}
Fix a countable cardinality $\kappa$ and suppose $X$ is a Banach space containing a $(C,D)$-good $\ell_\infty$-tree $(y_{\sigma_i})_{i=0}^\infty$ of height $\max W$ for some $C,D>0$.  Then there is a bi-Lipschitz embedding $\psi\colon T_{W,\kappa}\to X$ such that for all $u,v\in V(T_{W,\kappa})$,
\[\frac{1}{3D(1+D)}d(u,v)\leq \|\psi(u)-\psi(v)\|_X\leq Cd(u,v),\]
where $d$ is the shortest-path metric for $T_{W,\kappa}$, and furthermore, $\|\psi(u)-\psi(v)\|_X\geq d(u,v)/D$ when $u\Updownarrow v$.
\end{theorem}

\begin{proof}
Define the map $\psi\colon T_{W,\kappa}\to X$ by 
\[ \psi((r,A))= \sum_{B\preceq A}z(r,|B|) y_B\]
for every $(r,A)\in V(T_{W,\kappa})$.

Take any $u=(r,A)$ and $v=(s,B)$ in $V(T_{W,\kappa})$ and suppose first that $u$ and $v$ are adjacent with $A\preceq B$.
Then,
\begin{align*}
\|\psi(u)-\psi(v)\|_X&=\left\|\sum_{E\preceq A}z(r,|E|)y_E-\sum_{E\preceq B}z(s,|E|)y_E\right\|_X\\
&=\left\|\sum_{E\preceq B}(z(r,|E|)-z(s,|E|))y_E\right\|_X\\
&\leq C\max_{0\leq i\leq w_{s}}\{|z(r,i)-z(s,i)|\}\\
&\leq C.
\end{align*}
The triangle inequality applied to shortest paths then shows that $\|\psi(u)-\psi(v)\|_X\leq Cd(u,v)$ for all $u,v\in V(T_{W,\kappa})$.

For the left-hand inequality, take any $u=(r,A)$ and $v=(s,B)$ in $V(T_{W,\kappa})$, and suppose first that $u\Updownarrow v$.  Then
\[
\|\psi(u)-\psi(v)\|_X \geq  \frac{1}{D}|z(r,0)-z(s,0)|=\frac{1}{D}|r-s|=\frac{1}{D}d(u,v).\]

Suppose now that $u\ \cancel{\Updownarrow}\ v$.
Note that $n(u,v)=x(r,|A\wedge B|+1)=x(s,|A\wedge B|+1)$ and $m(u,v)=y(r,|A\wedge B|+1)=y(s,|A\wedge B|+1)$.
Let $n_1\in \mathbb{N}_0$ be such that $\sigma_{n_1}\preceq A$ and $|\sigma_{n_1}|=|A\wedge B|+1$.
Similarly, let $n_2\in \mathbb{N}_0$ be such that $\sigma_{n_2}\preceq B$ and $|\sigma_{n_2}|=|A\wedge B|+1$.
For $i\in \mathbb{N}_0$, let \[\alpha_i=
\begin{cases}
z(r,|\sigma_i|)-z(s,|\sigma_i|) & \sigma_i \preceq A\wedge B\\
z(r,|\sigma_i|) & A\wedge B\prec \sigma_i \preceq A\\
-z(s,|\sigma_i|) & A\wedge B \prec \sigma_i \preceq B\\
0 & \mbox{otherwise}.
\end{cases}
\]
Then, by Proposition \ref{dist},
\begin{align*}
\left\|\psi(u)-\psi(v)\right\|_X &= \left\|\sum_{i=0}^\infty \alpha_i y_{\sigma_i}\right\|_X\\
&\geq \frac{1}{D}\max \left\{\left\|\sum_{i=0}^{n_1}\alpha_iy_{\sigma_i}\right\|_X,\left\|\sum_{i=1}^{n_2}\alpha_iy_{\sigma_i}\right\|_X, \|\alpha_0y_{\emptyset}\|_X\right\}\\
&\geq \frac{1}{D(1+D)}\max\{|\alpha_{n_1}|,|\alpha_{n_2}|, |\alpha_0|\}\\
&
= \frac{1}{D(1+D)}\max\{z(r, |A\wedge B|+1),z(s,|A\wedge B|+1),|z(r,0)-z(s,0)|\}\\
& \begin{aligned}
= \frac{1}{D(1+D)} \max\{& \min\{r-n(u,v),m(u,v)-r\},\\
& \min\{s-n(u,v),m(u,v)-s\}, |r-s|\}
\end{aligned}\\
& \geq \frac{1}{3D(1+D)}d(u,v).
\qedhere
\end{align*}
\end{proof}

\begin{remark}
In the proof above, $1+D$ appears by using the triangle inequality and the fact that $D$ is a monotonicity constant for the sequence $(y_{\sigma_i})_{i=0}^\infty$.  One may replace $1+D$ with $D$ if $D$ is actually a \emph{bimonotonicity} constant.
\end{remark}

Theorem \ref{thm:ell_inf} generalizes Theorem 3.1 in \cite{baudier_etal}.  It was also shown in \cite{baudier_etal} that any reflexive Banach space with an unconditional asymptotic structure which is not asymptotically uniformly convexifiable will contain $(1+\varepsilon,1+\varepsilon)$-good $\ell_\infty$-trees of arbitrary height, for any $\varepsilon>0$.  Thus, Theorem \ref{thm:ell_inf} yields the following corollary.

\begin{corollary}
\label{cor:auc}
For any $\varepsilon>0$, every countably-branching bundle graph is bi-Lipschitzly embeddable with distortion bounded above by $6+\varepsilon$ into any reflexive Banach space with an unconditional asymptotic structure which is not asymptotically uniformly convexifiable.
\end{corollary}

Finally, in \cite{baudier_etal} it was shown that if a Banach space $X$ is asymptotically midpoint uniformly convexifiable, then no family of bundle graphs with nontrivial (meaning there is a vertex with nonzero depth) $\aleph_0$-branching base graph is equi-bi-Lipschitzly embeddable into $X$.  This fact combined with Corollary \ref{cor:auc} actually shows that the non-equi-bi-Lipschitz embeddability of any family of bundle graphs generated by a nontrivial $\aleph_0$-branching bundle graph characterizes asymptotic uniform convexifiability within the class of reflexive Banach spaces with an asymptotic unconditional structure.  We'll recall the definition of a family of bundle graphs generated from a base graph in Section \ref{sec:oslash}.

\section{Embedding into $L_1$}
\label{sec:L1}
In this section we show that for any countable cardinality $\kappa$, $T_{W,\kappa}$ is bi-Lipschitzly embeddable into $L_1$ (we'll use $L_1([0,M+1])$) with distortion bounded above by 2.  For each $v\in V(T_{W,\kappa})$, we'll map $v$ to the characteristic function of some set.  To get the distortion we desire, we need to make sure that the symmetric differences of the sets involved are large enough in Lebesgue measure $\lambda$.  The construction is somewhat technical, but is essentially done through intersections of supports of independent Bernoulli random variables.

Let $D$ be any common multiple of the numbers in $[M+1]\setminus\{0\}$. Let $\mathcal{F}$ be the family of finite unions of open subintervals of $[0,M+1]$ such that for each $P\in \mathcal{F}$, there is $N\in \mathbb{N}_0$ such that each maximal (with respect to set containment) subinterval of $P$ is equal to $\left(\frac{qD}{D^{N+1}},\frac{qD+r}{D^{N+1}}\right)$ for some $q\in[(M+1)D^{N}-1]$ and $r\in[D]$; and given $P\in \mathcal{F}$, let $N(P)$ be the minimum of such $N$ associated with $P$.  Note here that if $P,P'\in \mathcal{F}$ are such that $N(P)<N(P')$, then every subinterval of $P'$ is either contained in a subinterval of $P$ or has empty intersection with $P$.

Let $(\sigma_i)_{i=0}^\infty$ be an enumeration of $\aleph_0^{\leq \max W}$ and let $(P_i)_{i=0}^\infty$ be an enumeration of $\mathcal{F}$ such that $P_0=\emptyset$ and $N(P_i)\leq i$ for all $i\in \mathbb{N}_0$.  For each $i,j\in \mathbb{N}_0$, let $\theta(i,j)=2^i3^j-N(P_j)$.  Define $f\colon \mathbb{\aleph}_0^{\leq \max W}\times \mathcal{F}\times [D]\to \mathcal{F}$ by 
\[f(\sigma_i,P_j,k)=\bigcup_{\ell=0}^n\bigcup_{m=0}^{D^{\theta(i,j)}-1}\left(\alpha_\ell+\frac{mD(\beta_\ell-\alpha_\ell)}{D^{\theta(i,j)+1}},\alpha_\ell+\frac{(mD+k)(\beta_\ell-\alpha_\ell)}{D^{\theta(i,j)+1}}\right)\]
whenever $P_j=\bigsqcup_{\ell=0}^n (\alpha_\ell,\beta_\ell)$ (where $\sqcup$ means disjoint union), for all $i,j\in\mathbb{N}_0$ and $k\in[D]$.

We'll list the properties we need from $f$ in the following lemma, but one may think of $f(A,P,k)$ as the intersection of $P$ with the support of a Bernoulli random variable that has probability of success equal to $k/D$.  And if $B\neq A$, then $f(B,P,k)$ does the same thing, but with a random variable that is independent from that used for $f(A,P,k)$.

\begin{lemma}
\label{set theory}
The following hold for all $i,j\in \mathbb{N}_0$ and $k\in [D]$:
\begin{enumerate}[(i)]
\item $f(\sigma_i, P_j, k)\subseteq f(\sigma_i,P_j,k')\subseteq P$ if $k'\in[D]$ is such that $k\leq k'$.
\item $2^i3^j-1\leq N(f(\sigma_i,P_j,k))\leq 2^i3^j+2$ if $j\neq 0$ and $k\neq 0$.
\item $\lambda f(\sigma_i,P_j,k)=\frac{k}{D}\lambda (P_j)$.
\item $\lambda (P\cap f(\sigma_i,P_j,k))=\frac{k}{D}\lambda (P\cap P_j)$ if $P\in\mathcal{F}$ is such that $N(P)< N(P_j)$.
\end{enumerate}
\end{lemma}

\begin{proof}
(i):   This is obvious from the definition of $f$. \newline
(ii):  Every maximal subinterval of $P_j$ has length less than or equal to $1/D^{N\left(P_j\right)}$, so by the definition of $f$, every maximal subinterval of $f(\sigma_i,P_j,k)$ has length less than or equal to 
\[\frac{k}{D^{\theta(i,j)+1}}\cdot\frac{1}{D^{N\left(P_j\right)}}
= \frac{k}{D^{2^i3^j +1}}
\leq\frac{1}{D^{2^i3^j }}.\]  This means $N(f(\sigma_i,P_j,k))\geq 2^i3^j-1$.  Similarly, every (nontrivial) maximal subinterval of $P_j$ has length greater than or equal to $1/D^{N\left(P_j\right)+1}$, and so every (nontrivial) maximal subinterval of $f(\sigma_i,P_j,k)$ has length greater than or equal to \[\frac{k}{D^{\theta(i,j)+1}}\cdot \frac{1}{D^{N\left(P_j\right)+1}}=\frac{k}{D^{2^i3^j+2}}\geq \frac{1}{D^{2^i3^j+2}}.\]  Therefore $N(f(\sigma_i,P_j,k))\leq 2^i3^j+2$. \newline
(iii): Supposing $P_j=\bigsqcup_{\ell=0}^n(\alpha_\ell,\beta_\ell)$, then 
\[\lambda(f(\sigma_i,P_j,k))=\sum_{\ell=0}^n\sum_{m=0}^{D^{\theta(i,j)}-1}\frac{k(\beta_\ell-\alpha_\ell)}{D^{\theta(i,j)+1}}
=\frac{k}{D}\sum_{\ell=0}^n (\beta_\ell-\alpha_\ell)
=\frac{k}{D}\lambda(P_j).\]

(iv):  Since $N(P)< N(P_j)$, every subinterval of $P_j$ is either contained in a subinterval of $P$ or has empty intersection with $P$.
Thus, if $P_j=\bigsqcup_{\ell=0}^n(\alpha_\ell,\beta_\ell)$ and $I=\{\ell\in [n]\ |\ (\alpha_\ell,\beta_\ell)\subseteq P\}$, then 
\[\lambda(P\cap f(\sigma_i,P_j,k))=\sum_{\ell\in I}\sum_{m=0}^{D^{\theta(i,j)}-1}\frac{k(\beta_\ell-\alpha_\ell)}{D^{\theta(i,j)+1}}
=\frac{k}{D}\sum_{\ell\in I} (\beta_\ell-\alpha_\ell)
=\frac{k}{D}\lambda(P\cap P_j). \qedhere\]
\end{proof}

We are now ready to define the sets needed for our bi-Lipschitz embedding.  This is done by recursively defining the sets based on the depths of the vertices in our bundle graph.  At any given depth we'll construct the sets out of subsets of the sets that were defined for the previous depth.  If a vertex has height $r$ and depth $0$, we'll assign the set $[0,r]$ to this vertex.  Suppose there are vertices at heights $r\leq s$ with depth $0$ and $v$ is a vertex with depth $1$ at height  halfway between $r$ and $s$.  We'll assign a set of measure $r+(s-r)/2$ to $v$ by including $[0,r]$ with half of the set $[0,s]\setminus [0,r]$.  So for instance, we might assign the set $[0,r]\cup [r,(s-r)/2]$ to $v$.  However, there will be another vertex with depth $1$ at the same height as $v$.  For this vertex, we need to assign a different set of measure $r+(s-r)/2$, so we'll take half of the set $[0,s]\setminus [0,r]$ in a way that is independent of the way we did it with $v$.  For instance, we might use $[0,r]\cup [r,(s-r)/4]\cup [(s-r)/2,3(s-r)/4]$.  A similar process for all depths is used until every vertex has a subset of $[0,M+1]$ assigned to it with measure equal to its height.  We'll use the function $f$ defined above to take care of the independent selection of sets.  At this point we'll fix for the rest of the section a countable cardinality $\kappa$.  The formal construction follows.

Given $v=(r,A)\in V(T_{W,\kappa})$, define the sets $S_x(v,i)$ and $S_y(v,i)$ in $\mathcal{F}$ for $i\in [w_r]$ recursively by
\begin{align*}
S_x(v,0)&=[0,x(r,1)]\setminus [x(r,1)],\\
S_y(v,0)&=[0,y(r,1)]\setminus [y(r,1)],\\
\end{align*}
and
\begin{align*}
S_x(v,i)&=f\left(A\restriction_i,S_y(v,i-1)\setminus \mathrm{clos}(S_x(v,i-1)),\frac{x(r,i+1)-x(r,i)}{y(r,i)-x(r,i)}D\right),\\
S_y(v,i)&=f\left(A\restriction_i,S_y(v,i-1)\setminus \mathrm{clos}(S_x(v,i-1)),\frac{y(r,i+1)-x(r,i)}{y(r,i)-x(r,i)}D\right),
\end{align*}
for $i\in[w_r]\setminus\{0\}$.  Finally, let $S(v)=\mathrm{clos}\left(\bigcup_{i=0}^{w_r}S_x(v,i)\right)$.

\begin{lemma}
\label{differences}
Fix $v=(r,A)\in V(T_{W,\kappa})$.  The following hold for all $i\in [w_r]$.
\begin{enumerate}[(i)]
\item $\lambda (S_x(v,i))=x(r,i+1)-x(r,i)$.
\item $\lambda (S_y(v,i))=y(r,i+1)-x(r,i)$.
\item $\lambda (S_y(v,i)\setminus S_x(v,i))= y(r,i+1)-x(r,i+1)$.
\item $S_x(v,i)\cap S_x(v,i')=\emptyset$ if $i'\in [w_r]$ is such that $i'\neq i$.
\item $\lambda (\cup_{k=0}^i S_x(v,k))=x(r,i+1)$.
\end{enumerate}
\end{lemma}

\begin{proof}
(i)-(iii):  These statements certainly hold true for $i=0$.  And by (simultaneous) induction and Lemma \ref{set theory} (i) and (iii), they hold true for all $i\in [w_r]$. \newline
(iv):  By Lemma \ref{set theory} (i), $S_x(v,w_r-k)\subseteq S_y(v,w_r-k-1)\setminus S_x(v,w_r-k-1)$, and so $S_x(v,w_r-k)\cap S_x(v,w_r-k-1)=\emptyset$ for all $k\in [w_r-1]$.  In the same way, $S_y(v,w_r-k-1)\subseteq S_y(v,w_r-k-2)\setminus S_x(v,w_r-k-2)$, and so the first set inclusion implies $S_x(v,w_r-k)\cap S_x(v,w_r-k-2)=\emptyset$ for all $k\in [w_r-2]$.  Inductively, $S_x(v,w_r-k_1)\cap S_x(v,w_r-k_2)=\emptyset$ for all $k_1<k_2\in [w_r]$.\newline
(v):  This follows from parts (i) and (iv).
\end{proof}

\begin{lemma}
\label{independence}
Fix $u=(r,A)$ and $v=(s,B)$ in $V(T_{W,\kappa})$.  Then
\[\lambda (S(u) \cap S(v))=
\begin{cases}
\min\{r,s\} & u\Updownarrow v\\
n(u,v)+\frac{(r-n(u,v))(s-n(u,v))}{m(u,v)-n(u,v)} & u\ \cancel{\Updownarrow}\ v
\end{cases}\]
\end{lemma}

\begin{proof}
Let $K=|A\wedge B|$.  Suppose first that $u\Updownarrow v$ and $r\leq s$.  Then $y(r,K+1)\leq x(s,K+1)$.  Let $n\in [K]$ be such that $x(s,i)< y(r,i)$ (which implies $x(s,i)=x(r,i)$ and $y(s,i)=y(r,i)$) for all $i\in [n]$ while $y(r,n+1)\leq x(s,n+1)$.  An easy induction argument shows that $S_x(u,i)=S_x(v,i)$ and $S_y(u,i)=S_y(v,i)$ for all $i\in [n-1]$, and $S_y(u,n)\subseteq S_x(v,n)$.  Another easy induction argument and Lemma \ref{set theory} (i) shows $S_x(u,i)\subseteq \bigcup_{j=0}^nS_x(v,j)$ for all $i\in [w_r]$ by , and so Lemma \ref{differences} (v) yields
\[
\lambda (S(u)\cap S(v))=\lambda \left(\bigcup_{i=0}^{w_r}S_x(u,i)\cap \bigcup_{j=0}^{w_s}S_x(v,j)\right)
=\lambda\left(\bigcup_{i=0}^{w_r}S_x(u,i)\right)
=x(r,w_r+1)
=r.
\]

Suppose now that $u\ \cancel{\Updownarrow}\ v$.  Then $x(r,i)=x(s,i)$ and $y(r,i)=y(s,i)$ for all $i\in [K+1]$, and so $S_x(u,i)=S_x(v,i)$ and $S_y(u,i)=S_y(v,i)$ for all $i\in[K]$, by an easy induction argument.  Note that $x(r,K+1)=x(s,K+1)=n(u,v)$ and $y(r,K+1)=y(s,K+1)=m(u,v)$.  For any $i\in [w_r]\setminus [K]$ and $j\in [w_s]\setminus [K]$, repeated applications of Lemma \ref{set theory} (ii) and (iv) and then Lemma \ref{differences} (iii) show
\begin{align*}
\lambda (S_x(u,i)&\cap S_x(v,j))\\
&= \frac{x(r,i+1)-x(r,i)}{y(r,K+1)-x(r,K+1)}\cdot \frac{x(s,j+1)-x(s,j)}{y(s,K+1)-x(s,K+1)}\cdot \lambda (S_y(u,K)\setminus S_x(u,K))\\
&=\frac{(x(r,i+1)-x(r,i))(x(s,j+1)-x(s,j))}{m(u,v)-n(u,v)}.
\end{align*}
This with Lemma \ref{differences} (iv) and (v) implies
\begin{align*}
\lambda (S(u)\cap S(v))&=\lambda \left(\bigcup_{i=0}^{w_r}S_x(u,i)\cap \bigcup_{j=0}^{w_s}S_x(v,j)\right)\\
&=x(r,K+1)+\sum_{i=K+1}^{w_r}\sum_{j=K+1}^{w_s}\frac{(x(r,i+1)-x(r,i))(x(s,j+1)-x(s,j))}{m(u,v)-n(u,v)}\\
&=x(r,K+1)+\frac{(x(r,w_r+1)-x(r,K+1))(y(s,w_s+1)-y(s,K+1))}{m(u,v)-n(u,v)}\\
&=n(u,v)+\frac{(r-n(u,v))(s-n(u,v))}{m(u,v)-n(u,v)}.\qedhere
\end{align*}

\end{proof}

\begin{theorem}
There is a bi-Lipschitz embedding $\psi\colon T_{W,\kappa}\to L_1$ such that for all $u,v\in V(T_{W,\kappa})$,
\[\frac{1}{2}d(u,v)\leq \|\psi(u)-\psi(v)\|_{L_1}\leq d(u,v),\]
where $d$ is the shortest-path metric for $T_{W,\kappa}$, and furthermore $\|\psi(u)-\psi(v)\|_{L_1}=d(u,v)$ when $u\Updownarrow v$.
\end{theorem}

\begin{proof}

Define the map $\psi\colon T_{W,\kappa}\to L_1$ by $\psi(v)=\chi_{S(v)}$ for every $v\in V(T_{W,\kappa})$.

Take any $u=(r,A)$ and $v=(s,B)$ in $V(T_{W,\kappa})$ and suppose first that $u\Updownarrow v$.
By Lemma \ref{differences} (v) and Lemma \ref{independence},
\[\|\psi(u)-\psi(v)\|_{L_1}=\lambda (S(u))+\lambda(S(v))-2\lambda (S(u)\cap S(v))
=r+s-2\min\{r,s\}
=|r-s|
=d(u,v).
\]
Lemma \ref{path} and the triangle inequality applied to shortest paths then shows that $\|\psi(u)-\psi(v)\|_{L_1}\leq d(u,v)$ for all $u,v\in V(T_{W,\kappa})$.

Suppose now that $u\ \cancel{\Updownarrow}\ v$.  By Lemma \ref{differences} (v) and Lemma \ref{independence},
\begin{align*}
\|\psi(u)-\psi(v)\|_{L_1} &= \lambda(S(u))+\lambda(S(v))-2\lambda (S(u)\cap S(v))\\
&= r+s-2n(u,v)-2\frac{(r-n(u,v))(s-n(u,v))}{m(u,v)-n(u,v)}\\
&=\alpha +\beta -2\frac{\alpha \beta}{\gamma},
\end{align*}
where $\alpha=r-n(u,v)$, $\beta=s-n(u,v)$, and $\gamma=m(u,v)-n(u,v)$.
Suppose first that $\max\{\alpha,\beta\}\leq \frac{1}{2}\gamma$.
Then by Proposition \ref{dist} and the above,
\[
\|\psi(u)-\psi(v)\|_{L_1}\geq \alpha +\beta-\min\{\alpha,\beta\}\\
=\max\{\alpha, \beta\}\\
\geq \frac{1}{2}d(u,v).
\]

Suppose next that $\min\{\alpha,\beta\}\leq \frac{1}{2}\gamma\leq\max\{\alpha,\beta\}$.
Then by Proposition \ref{dist} and the above,
\[
\|\psi(u)-\psi(v)\|_{L_1} =
\min\{\alpha,\beta\}+\max\{\alpha,\beta\}\left(1-2\frac{\min\{\alpha,\beta\}}{\gamma}\right)\\
\geq  \frac{1}{2}\gamma \\
 \geq \frac{1}{2}d(u,v).
\]

Finally, suppose $\frac{1}{2}\gamma\leq \min\{\alpha, \beta\}$.
Then by Proposition \ref{dist} and the above,
\begin{align*}
\|\psi(u)-\psi(v)\|_{L_1}
&=\frac{1}{2\gamma}(2\alpha \gamma +2\beta \gamma -4\alpha \beta)\\
& =\frac{1}{2\gamma}(\gamma^2-(2\alpha-\gamma)(2\beta-\gamma))\\
&\geq \frac{1}{2\gamma}(\gamma^2-(\alpha+\beta-\gamma)\gamma)\\
& =\frac{1}{2}(2\gamma-(\alpha+\beta)) \\
&=\frac{1}{2}d(u,v). \qedhere
\end{align*}

\end{proof}

\section{Embedding into Banach spaces with ESA bases}
\label{sec:superreflexivity}
In this section we show that for any finite cardinality $\kappa$, $T_{W,\kappa}$ is bi-Lipschitzly embeddable into any Banach space with an ESA basis with distortion bounded above by a constant depending only on $W$.  

\begin{definition}
\label{def:esa}
Let $(X,\|\cdot\|_X)$ be a Banach space.
\begin{enumerate}[(i)]
\item A sequence $(e_n)_{n=1}^\infty\subseteq X$ is said to be \emph{equal-signs-additive (ESA)} if for all $(a)_{n=1}^\infty\in c_{00}$ and $k\in \mathbb{N}$ such that $a_ka_{k+1}\geq 0$,
\[\left\|\sum_{n=1}^{k-1}a_ne_n+(a_k+a_{k+1})e_k+\sum_{n=k+2}^\infty a_ne_n\right\|_X=\left\|\sum_{n=1}^\infty a_ne_n\right\|_X.\]
\item A sequence $(e_n)_{n=1}^\infty\subseteq X$ is said to be \emph{subadditive (SA)} if for all $(a)_{n=1}^\infty\in c_{00}$,
\[\left\|\sum_{n=1}^{k-1}a_ne_n+(a_k+a_{k+1})e_k+\sum_{n=k+2}^\infty a_ne_n\right\|_X\leq\left\|\sum_{n=1}^\infty a_ne_n\right\|_X.\]
\item A sequence $(e_n)_{n=1}^\infty\subseteq X$ is said to be \emph{invariant under spreading (IS)} if for all $(a)_{n=1}^\infty\in c_{00}$ and increasing sequences $(k_n)_{n=1}^\infty\subseteq \mathbb{N}$,
\[\left\|\sum_{n=1}^\infty a_ne_{k_n}\right\|=\left\|\sum_{n=1}^\infty a_ne_n\right\|.\]
\end{enumerate} 
\end{definition}

In \cite{brunel_sucheston}, A. Brunel and L. Sucheston show that a sequence is ESA if and only if it is SA, and that every ESA sequence is also an IS basis for its linear span.  We will use these facts without mention.  More information about ESA sequences can be found in \cite{brunel_sucheston} and \cite{argyros_motakis_sari}.

To construct the embedding, we follow roughly the same procedure as used in the previous section.  However, with $L_1$ we were able to subdivide the interval $[0,M+1]$ as finely as needed to accommodate the bundle graph.  That is, we could use the existence of infinitely many independent Bernoulli random variables.  If instead of $L_1$, we try to embed into a general Banach space with a basis, we still need independent Bernoulli random variables to choose the support of an embedded vertex, but the random variables are now discrete.  In other words, the more vertices in our graph we have to embed, the further down the basis we have to go if we want to mimic the procedure used for $L_1$.  This and the fact that we don't have an explicitly defined norm anymore make the argument more subtle.

We fix now for the rest of this section a finite cardinality $\kappa$ and let $\mu=\left|\kappa^{\leq \max W}\right|\in \mathbb{N}$.  We also fix an independent collection $\{Y_i\}_{i=1}^\mu$ of Bernoulli random variables defined on $[2^\mu]\setminus \{0\}$ (equipped with the uniform probability measure) with probability of success equal to $1/2$.
Concretely, for each $i\in [\mu]\setminus\{0\}$, we may define $Y_i\colon \left[2^{\mu}\right]\setminus\{0\} \to \{0,1\}$ by 
\[Y_i(j)=\begin{cases}
1 & j\equiv n \pmod{2^{\mu-(i-1)}} \mbox{ for some } n\in [2^{\mu-i}]\setminus\{0\}\\
0 & \mbox{otherwise} 
\end{cases}\] 
for all $j\in \left[2^\mu\right]\setminus \{0\}$.   
For each $j\in\mathbb{N}$, let $I_j=[j(M+1)]\setminus [(j-1)(M+1)]$, and let $\mathcal{F}_j$ be the family of subsets of $I_j$ such that for each $P\in \mathcal{F}_j$, either $P=\emptyset$, or $P\neq \emptyset$ and $|P|=\max (P)-\min (P)+1$ (which implies $P$ has no ``gaps'').  That is, we break up the natural number line into blocks of size $M+1$ and let $\mathcal{F}_j$ be the family of intervals contained in the the $j$-th block.

Let $\{\sigma_i\}_{i=0}^\mu$ be an enumeration of $\kappa^{\leq \max W}$.  For $A\in \kappa^{\leq \max W}$ and $i\in [\mu]$, we'll let $Y_A=Y_i$ if $A=\sigma_i$.  Define for each $j\in [2^\mu]\setminus \{0\}$ the function $f_j\colon \kappa^{\leq \max W}\times \mathcal{F}_j\times [M] \to \mathcal{F}_j$ by  
\[f_j(\sigma_i,P,k)=I_j\cap \{Y_i(j)(\inf P+\ell), (1-Y_i(j))(\sup P-\ell)\}_{\ell=0}^{k-1}\]
for all $i\in [\mu]$, $P\in\mathcal{F}_j$, and $k\in [M]$.  

To summarize what is happening, we assign independent Bernoulli random variables to the elements of $\kappa^{\leq \max W}$.  Given $A\in \kappa^{\leq \max W}$ with its assigned random variable $Y_A$ and an interval $P$ in $I_j$, $f_j(A,P,k)$ will take the first $k$ elements of $P$ in $I_j$ if $Y_A(j)=1$ and the last $k$ elements if $Y_A(j)=0$.  Note that $f_j(\sigma_i,P,k)\subseteq P$ if $k\leq |P|$.  Thus, if $P$ is a union of subintervals of the $I_j's$, we can use the $f_j's$ to simultaneously select $k$ elements from $P$ out of each interval $P\cap I_j$, and these selections will be independent for different elements of $\kappa^{\leq \max W}$.  This is quite analogous to what happened in the last section.  The construction of the supports of our embedded vertices is likewise similar, but in the end we can't just map a vertex to a characteristic function.  We'll have to modify slightly in order to use the ESA property to obtain a good distortion.

Given $v=(r,A)\in V(T_{W,\kappa})$ and $j\in [2^\mu]\setminus \{0\}$, define the sets $S_{x,j}(v,i)$ and $S_{y,j}(v,i)$ in $\mathcal{F}_j$ for $i\in[w_r]$ recursively by   
\begin{align*}
S_{x,j}(v,0)&=[(j-1)(M+1)+x(r,1)]\setminus [(j-1)(M+1)],\\
S_{y,j}(v,0)&=[(j-1)(M+1)+y(r,1)]\setminus [(j-1)(M+1)],\\
\end{align*}
and
\begin{align*}
S_{x,j}(v,i)&=f_j\left(A\restriction_i,S_{y,j}(v,i-1)\setminus S_{x,j}(v,i-1),x(r,i+1)-x(r,i)\right),\\
S_{y,j}(v,i)&=f_j\left(A\restriction_i,S_{y,j}(v,i-1)\setminus S_{x,j}(v,i-1),y(r,i+1)-x(r,i)\right),
\end{align*}
for $i\in [w_r]\setminus \{0\}$;

and then define 

\[S_j(v)=\bigcup_{i=0}^{w_r}S_{x,j}(v,i).\]
Finally, let
\begin{align*}
S_{j,+}(v)&= \left\{(j-1)(M+1)+n\ |\ n\in S_j(v)\right\},\\
S_{j,-}(v)&= \left\{(3j-1)(M+1)-n\ |\ n\in S_j(v)\right\}.
\end{align*}

$S_{j,+}$ will take a copy of $S_j(V)$ and put it in $I_{2j-1}$.  $S_{j,-}$ will also take a copy of $S_j(V)$ and put it in $I_{2j}$.  The copy for $S_{j,-}$, however, is backwards.  That is, $S_{j,-}$ is a reflection of $S_{j,+}$ across the middle of $I_{2j-1}\cup I_{2j}$.  The purpose of this is to allow us to take advantage of the subadditivity of an ESA basis later.

\begin{lemma}
\label{difference2}
Fix $v=(r,A)\in V(T_{W,\kappa})$.  The following hold for all $i\in[w_r]$ and  $j\in [2^\mu]\setminus\{0\}$,
\begin{enumerate}[(i)]
\item $|S_{x,j}(v,i)|=x(r,i+1)-x(r,i)$.
\item $|S_{y,j}(v,i)|=y(r,i+1)-x(r,i)$.
\item $|S_{y,j}(v,i)\setminus S_{x,j}(v,i)|= y(r,i+1)-x(r,i+1)$.
\item $S_{x,j}(v,i)\cap S_{x,j}(v,i')=\emptyset$ if $i'\in [w_r]$ is such that $i'\neq i$.
\item $|\bigcup_{k=0}^i S_{x,j}(v,k)|=x(r,i+1)$.
\item $|S_{j,+}(v)|=|S_{j,-}(v)|=r$.
\item $S_{j,+}(v)\subseteq I_{2j-1}$ and $S_{j,-}(v)\subseteq I_{2j}$.
\item If $i'\in[w_r]\setminus [i]$ and $Y_{A\restriction_n}(j)=1$ for all $n\in [i']\setminus [i]$, then
\begin{enumerate}[(a)]
\item $\bigcup_{k=i+1}^{i'} S_{x,j}(v,k)=I_j\cap \{\inf S_{y,j}(v,i)\setminus S_{x,j}(v,i)+\ell\}_{\ell=0}^{x(r,i'+1)-x(r,i+1)-1}$.
\item $ S_{y,j}(v,i')\setminus S_{x,j}(v,i')=I_j\cap \{\inf S_{y,j}(v,i)\setminus S_{x,j}(v,i)+\ell\}_{\ell=x(r,i'+1)-x(r,i+1)}^{y(r,i'+1)-x(r,i+1)-1}$
\end{enumerate}
\item If $i'\in[w_r]\setminus [i]$ and $Y_{A\restriction_n}(j)=0$ for all $n\in [i']\setminus [i]$, then
\begin{enumerate}[(a)]
\item $\bigcup_{k=i+1}^{i'}S_{x,j}(v,k)=I_j\cap \{\sup S_{y,j}(v,i)\setminus S_{x,j}(v,i)-\ell\}_{\ell=0}^{x(r,i'+1)-x(r,i+1)-1}$.
\item $S_{y,j}(v,i')\setminus S_{x,j}(v,i')=I_j\cap \{\sup S_{y,j}(v,i)\setminus S_{x,j}(v,i)-\ell\}_{\ell=x(r,i'+1)-x(r,i+1)}^{y(r,i'+1)-x(r,i+1)-1}$.
\end{enumerate}
\end{enumerate}
\end{lemma}

\begin{proof}
(i)-(iii):  These statements certainly hold true for $i=0$.  And by (simultaneous) induction, they hold true for all $i\in [w_r]$. \newline
(iv):  We have $S_{x,j}(v,w_r-k)\subseteq S_{y,j}(v,w_r-k-1)\setminus S_{x,j}(v,w_r-k-1)$, and so $S_{x,j}(v,w_r-k)\cap S_{x,j}(v,w_r-k-1)=\emptyset$ for all $k\in [w_r-1]$.  In the same way, $S_{y,j}(v,w_r-k-1)\subseteq S_{y,j}(v,w_r-k-2)\setminus S_{x,j}(v,w_r-k-2)$, and so the first set inclusion implies $S_{x,j}(v,w_r-k)\cap S_{x,j}(v,w_r-k-2)=\emptyset$ for all $k\in [w_r-2]$.  Inductively, $S_{x,j}(v,w_r-k_1)\cap S_{x,j}(v,w_r-k_2)=\emptyset$ for all $k_1<k_2\in [w_r]$.\newline
(v):  This follows from parts (i) and (iv).\newline
(vi)-(vii):  By definition, $S_j(v)\subseteq I_j$.  The rest follows from the definitions, part (v), and the fact that $n\mapsto (j-1)(M+1)+n$ is a bijection from $I_j$ to $I_{2j-1}$ and $n\mapsto (3j-1)(M+1)-n$ is a bijection from $I_j$ to $I_{2j}$ for each $j\in [2^\mu]\setminus\{0\}$.\newline
(viii)-(ix):  The statements are true for $i'=i+1$ by the definitions of $S_{x,j}(v,i+1)$ and $f_j$, and the statements hold for arbitrary $i'$ by a simple induction.
\end{proof}

At this point we're almost ready to define the embedding, but we need to be able to make sure that for enough $j\in[\mu]\setminus\{0\}$, the symmetric difference of $S_j(u)$ and $S_j(v)$ is large enough when $u\neq v\in V(T_{W,\kappa})$.  Unfortunately, the amount of $j\in[\mu]\setminus\{0\}$ we can do this for depends on $W$.  We define the parameter $p_W$ to be the minimum of all $p\in \mathbb{N}$ such that for all $r\in[M+1]$ and $i\in\mathbb{N}_0$,
\begin{itemize}
\item $x(r,i+p)\geq (x(r,i)+y(r,i))/2$ whenever $r\geq (x(r,i)+y(r,i))/2$.
\item $y(r,i+p)\leq (x(r,i)+y(r,i))/2$ whenever $r\leq (x(r,i)+y(r,i))/2$.
\end{itemize}
One may easily check that $p_W\leq \max W+1$.

\begin{lemma}
\label{lower_bound}
Fix $u=(r,A)$ and $v=(s,B)$ in $V(T_{W,\kappa})$ such that $u\ \cancel{\Updownarrow}\ v$ and $r\leq s$.  Let
\[ \mathcal{I}_{u,v}=\{ j\in [2^\mu]\setminus \{0\}\ |\ Y_{B\restriction_{|A\wedge B|+n}}(j)=1 \mbox{ and }Y_{A\restriction_{|A\wedge B|+n}}(j)=0 \mbox{ for all } n\in [p_W]\setminus \{0\}\}.\]
Then for each $j\in \mathcal{I}_{u,v}$, there is $\mathcal{L}_j\subseteq \bigcup_{k=|A\wedge B|+1}^{w_s}S_{x,j}(v,k)$ such that
\begin{enumerate}[(i)]
\item $|\mathcal{L}_j|\geq d(u,v)/2$, where $d$ is the shortest-path metric for $T_{W,\kappa}$.
\item $\max \mathcal{L}_j< \inf \bigcup_{k=|A\wedge B|+1}^{w_r}S_{x,j}(u,k)$.
\end{enumerate}
\end{lemma}

\begin{proof}
Let $K=|A\wedge B|$.  We have $x(r,i)=x(s,i)$ and $y(r,i)=y(s,i)$ for all $i\in [K+1]$, and so $S_{x,j}(u,i)=S_{x,j}(v,i)$ and $S_{y,j}(u,i)=S_{y,j}(v,i)$ for all $i\in [K]$, by an easy induction argument.  Note that $x(r,K+1)=x(s,K+1)=n(u,v)$ and $y(r,K+1)=y(s,K+1)=m(u,v)$, which implies $|S_{y,j}(v,K)\setminus S_{x,j}(v,K)|=m(u,v)-n(u,v)>0$, by Lemma \ref{difference2} (iii).  Note that $x(r,i)\leq x(s,i)$ and $y(r,i)\leq y(s,i)$ for all $i\in [M+1]$ by the fact that $r\leq s$.

Suppose first that $s\leq (m(u,v)+n(u,v))/2$ and let $\mathcal{L}_j=\bigcup_{k=K+1}^{w_s}S_{x,j}(v,k)$ for each $j\in \mathcal{I}_{u,v}$.  Then by definition of $p_W$, 
\begin{align*}
\min S_{y,j}(v,K)\setminus S_{x,j}(v,K)&+y(s,K+p_W+1)-n(u,v)-1\\
&\leq \min S_{y,j}(v,K)\setminus S_{x,j}(v,K)+\frac{m(u,v)-n(u,v)}{2}-1\\
&\leq \max S_{y,j}(u,K)\setminus S_{x,j}(u,K)-\left(\frac{m(u,v)-n(u,v)}{2}-1\right)\\
&\leq \max S_{y,j}(u,K)\setminus S_{x,j}(u,K)-(y(r,K+p_W+1)-n(u,v)-1),
\end{align*}
and so, by Lemma \ref{difference2} (viii) and (ix), 
$\mathcal{L}_j\cap\left(\bigcup_{k=K+1}^{w_r} S_{x,j}(u,k)\right)=\emptyset$ and therefore $\max \mathcal{L}_j< \inf \bigcup_{k=K+1}^{w_r}S_{x,j}(u,k)$ for each $j\in \mathcal{I}_{u,v}$.
Moreover, by Lemma \ref{difference2} (v) and Proposition \ref{dist}, \[|\mathcal{L}_j|=x(s,w_s+1)-x(s,K+1)=s-n(u,v)\geq d(u,v)/2\] for each $j\in \mathcal{I}_{u,v}$.

Suppose now that $r\geq (m(u,v)+n(u,v))/2$ and let \[\mathcal{L}_j=\left(\bigcup_{k=K+1}^{w_s}S_{x,j}(v,k)\right)\setminus  \left(\bigcup_{k=K+1}^{w_r}S_{x,j}(u,k)\right)\] for each $j\in \mathcal{I}_{u,v}$.  Then by definition of $p_W$,
\begin{align*}
\min S_{y,j}(v,K)\setminus S_{x,j}(v,K)&+x(s,K+p_W+1)-n(u,v)-1\\
&\geq \min S_{y,j}(v,K)\setminus S_{x,j}(v,K)+\frac{m(u,v)-n(u,v)}{2}-1\\
&\geq \max S_{y,j}(u,K)\setminus S_{x,j}(u,K)-\left(\frac{m(u,v)-n(u,v)}{2}-1\right)\\
&\geq \max S_{y,j}(u,K)\setminus S_{x,j}(u,K)-\left(x(r,K+p_W+1)-n(u,v)-1\right),
\end{align*}
and so, by Lemma \ref{difference2} (viii) and (ix), 
$\mathcal{L}_j\cup\left(\bigcup_{k=K+1}^{w_r} S_{x,j}(u,k)\right)=S_{y,j}(v,K)\setminus S_{x,j}(v,K)$ and therefore
, by Lemma \ref{difference2} (v) and Proposition \ref{dist}, \begin{align*}
|\mathcal{L}_j|&=|S_{y,j}(v,i)\setminus S_{x,j}(v,i)|-\left|\left(\bigcup_{k=K+1}^{w_r} S_{x,j}(u,k)\right)\right|\\
&=(m(u,v)-n(u,v))-(r-n(u,v))=m(u,v)-r\geq d(u,v)/2
\end{align*}
for each $j\in \mathcal{I}_{u,v}$.  Moreover, $\max \mathcal{L}_j< \inf \bigcup_{k=K+1}^{w_r}S_{x,j}(u,k)$ for each $j\in\mathcal{I}_{u,v}$.

Finally, suppose $r\leq (m(u,v)+n(u,v))/2\leq s$ and let \[\mathcal{L}_j=\{\min S_{y,j}(v,K)\setminus S_{x,K}(v,i)+\ell\}_{\ell=0}^{(m(u,v)-n(u,v))/2-1}\] for each $j\in \mathcal{I}_{u,v}$.
Then by definition of $p_W$,
\begin{align*}
\max \mathcal{L}_j
\leq \min\{ &\min S_{y,j}(v,K)\setminus S_{x,j}(v,K)+x(s,K+p_W+1)-n(u,v)-1,\\
& \max S_{y,j}(u,K)\setminus S_{x,j}(u,K)-\left(y(r,K+p_W+1)-n(u,v)-1\right)\},
\end{align*}
and so, by Lemma \ref{difference2} (viii) and (ix),
$\mathcal{L}_j\subseteq \bigcup_{k=K+1}^{w_s}S_{x,j}(v,k)$ and $\max \mathcal{L}_j< \inf \bigcup_{k=K+1}^{w_r}S_{x,j}(u,k)$ for each $j\in \mathcal{I}_{u,v}$.
Moreover, by Proposition \ref{dist}, $|\mathcal{L}_j|=(m(u,v)-n(u,v))/2\geq d(u,v)/2$ for each $j\in \mathcal{I}_{u,v}$..
\end{proof}

\begin{theorem}
\label{thm:esa}
Suppose $X$ is a Banach space with an ESA basis $(e_n)_{n=1}^\infty$.  Then there is a bi-Lipschitz embedding $\psi\colon T_{W,\kappa}\to X$ such that for all $u,v\in V(T_{W,\kappa})$,
\[\frac{1}{2^{2p_W+1}}d(u,v)\leq \|\psi(u)-\psi(v)\|_X\leq d(u,v),\]
where $d$ is the shortest-path metric for $T_{W,\kappa}$, and furthermore $\|\psi(u)-\psi(v)\|_X=d(u,v)$ when $u\Updownarrow v$.
\end{theorem}

\begin{proof}
Let $\eta=\left\|\sum_{j=1}^{2^\mu} e_{2j-1}-e_{2j}\right\|_X$ and define the map $\psi\colon T_{W,\kappa}\to X$ by \[\psi(v)=1/\eta\sum_{j=1}^{2^\mu}\left(\sum_{n\in S_{j,+}(v)}e_n-\sum_{n\in S_{j,-}(v)}e_n\right)\] for every $v\in V(T_{W,\kappa})$.

Take any $u=(r,A)$ and $v=(s,B)$ in $V(T_{W,\kappa})$.  Let $K=|A\wedge B|$ and suppose first that $u\Updownarrow v$ and $r\leq s$.  Then $y(r,K+1)\leq x(s,K+1)$.  Let $n\in [K]$ be such that $x(s,i)< y(r,i)$ (which implies $x(r,i)=x(s,i)$ and $y(r,i)=y(s,i)$) for all $i\in [n]$ while $y(r,n+1)\leq x(s,n+1)$.  An easy induction argument shows that $S_{x,j}(u,i)=S_{x,j}(v,i)$ and $S_{y,j}(u,i)=S_{y,j}(v,i)$ for all $i\in [n]$ and $j\in [2^\mu]\setminus \{0\}$; and $S_{y,j}(u,n)\subseteq S_{x,j}(v,n)$ for all $j\in [2^\mu]\setminus \{0\}$.  Another easy induction argument shows $S_{x,j}(u,i)\subseteq \bigcup_{k=0}^{n}S_{x,j}(v,k)$ for all $i\in [w_r]$ and $j\in [2^\mu]\setminus \{0\}$ (and so $S_{j,+}(u)\subseteq S_{j,+}(v)$ and $S_{j,-}(u)\subseteq S_{j,-}(v)$ for all $j\in [2^\mu]\setminus \{0\}$).  Thus, by Lemma \ref{dist}, Lemma \ref{difference2} (vi), and the assumption that the basis is ESA;
\begin{align*}
\|\psi(u)-\psi(v)\|_X&=\frac{1}{\eta}\left\|\sum_{j=1}^{2^\mu}\left(\left(\sum_{n\in S_{j,+}(u)}e_n-\sum_{n\in S_{j,-}(u)}e_n\right)-\left(\sum_{n\in S_{j,+}(v)}e_n-\sum_{n\in S_{j,-}(v)}e_n\right)\right)\right\|_X\\
&=\frac{1}{\eta}\left\|\sum_{j=1}^{2^\mu}\left(\sum_{n\in \cap S_{j,+}(v)\setminus S_{j,+}(u)}e_n-\sum_{n\in S_{j,-}(v)\setminus S_{j,-}(u)} e_n\right)\right\|_X\\
&=\frac{1}{\eta}\left\|\sum_{j=1}^{2^\mu}(s-r)(e_{2j-1}-e_{2j})\right\|_X\\
&=s-r\\
&=d(u,v).
\end{align*}
Lemma \ref{path} and the triangle inequality applied to shortest paths then shows that $\|\psi(u)-\psi(v)\|_X\leq d(u,v)$ for all $u,v\in V(T_{W,\kappa})$.

Suppose now that $u\ \cancel{\Updownarrow}\ v$ and $r\leq s$.  Define $\mathcal{I}_{u,v}$ as in Lemma \ref{lower_bound} and for each $j\in \mathcal{I}_{u,v}$, let $\mathcal{L}_j$ be chosen as in Lemma \ref{lower_bound}.
Note that, by independence of the Bernoulli random variables defined at the beginning of this section, $|\mathcal{I}_{u,v}|= 2^{\mu-\nu}$ for some $\nu\in [2p_W]$.  Let
\begin{align*}
\mathcal{L}_{j,+}&= \left\{(j-1)(M+1)+n\ |\ n\in \mathcal{L}_j\right\},\\
\mathcal{L}_{j,+}&= \left\{(3j-1)(M+1)-n\ |\ n\in \mathcal{L}_j\right\}.
\end{align*}
Recall that $n\mapsto (j-1)(M+1)+n$ is a bijection from $I_j$ to $I_{2j-1}$ and $n\mapsto (3j-1)(M+1)-n$ is a bijection from $I_j$ to $I_{2j}$ for each $j\in \mathcal{I}_{u,v}$.  Furthermore, the images of the two maps will be reflections of each other across the middle of $I_{2j-1}\cup I_{2j}$ when the maps are applied to the same set.
By Lemma \ref{difference2} (vi), Lemma \ref{lower_bound}, and the assumption that the basis is ESA;
\begin{align*}
\|\psi(u)-\psi(v)\|_X&=\frac{1}{\eta}\left\|\sum_{j=1}^{2^\mu}\left(\left(\sum_{n\in S_{j,+}(u)}e_n-\sum_{n\in S_{j,-}(u)}e_n\right)-\left(\sum_{n\in S_{j,+}(v)}e_n-\sum_{n\in S_{j,-}(v)}e_n\right)\right)\right\|_X\\
&\geq \frac{1}{\eta}\left\|\sum_{j\in \mathcal{I}_{u,v}}\left(\left(\sum_{n\in S_{j,+}(v)}e_n-\sum_{n\in S_{j,+}(u)}e_n\right)-\left(\sum_{n\in S_{j,-}(v)}e_n-\sum_{n\in S_{j,-}(u)}e_n\right)\right)\right\|_X\\
&\geq \frac{1}{\eta}\left\|\sum_{j\in \mathcal{I}_{u,v}}\left(\sum_{n\in \mathcal{L}_{j,+}}e_n-\sum_{n\in \mathcal{L}_{j,-}}e_n\right)\right\|_X\\
&\geq \frac{d(u,v)}{2\eta}\left\|\sum_{j\in\mathcal{I}_{u,v}}(e_{2j-1}-e_{2j})\right\|_X\\
&=\frac{d(u,v)}{2\eta}\cdot \frac{1}{2^\nu} \cdot \sum_{k=1}^{2^\nu} \left\|\sum_{j=(k-1)2^{\mu-\nu}+1}^{k2^{\mu-\nu}}(e_{2j-1}-e_{2j})\right\|_X\\
&\geq \frac{d(u,v)}{2^{\nu+1}\eta}\left\|\sum_{j=1}^{2^\mu} (e_{2j-1}-e_{2j})\right\|_X\\
&\geq \frac{1}{2^{2p_W+1}}d(u,v).\qedhere
\end{align*}
\end{proof}

We'll show in the next section that actually the entire family of bundle graphs generated by $T_{W,\kappa}$ is bi-Lipschitzly embeddable with the same distortion bound of $2^{2p_W+1}$ into a Banach space with an ESA basis.

\section{The $\oslash$-product}
\label{sec:oslash}
Given two $\kappa$-branching bundle graphs $G$ and $H$, we can define a new $\kappa$-branching bundle graph $G\oslash H$ by replacing every edge in $G$ with a copy of $H$ (where the bottom of $H$ is identified with the lower endpoint of the edge $H$ is replacing and the top of $H$ is identified with the higher endpoint).  For this section we fix another sequence $W'=\{w_{r}'\}_{r=0}^{M'+1}\subseteq \mathbb{N}_0$ such that $w_0'=w_{M'+1}'=0$.  We will show how to determine $W''$ so that $T_{W,\kappa}\oslash T_{W',\kappa}= T_{W'',\kappa}$.  Once $W''$ is found, we can use Theorem \ref{thm:esa} to find a bound on the worst distortion for a bi-Lipschitz embedding of $T_{W'',\kappa}$ into a Banach space with an ESA basis.  In particular, we'll show that the distortion bound found in Theorem \ref{thm:esa} is no worse for $T_{W,\kappa}\oslash T_{W,\kappa}$ than it was for $T_{W,\kappa}$, allowing us to generalize the characterizations of superreflexifity found in \cite{johnson_schechtman} and \cite{ostrovskii_randrianantoanina}.

Given a bundle graph $G=(V,E)$ and $n\in \mathbb{N}_0$, let $V_n=\{v\in V\ |\ \mathrm{height}(v)=n\}$ and let $E_n=\{\{u,v\}\in E\ |\ u\in V_n \mbox{ and } v\in V_{n+1}\}$.  If another bundle graph $H$ is given, we may create a new bundle graph $G\oslash_n H$ by replacing every edge in $E_n$ with $H$ for some $n$.  Explicitly, if $G=(V,E)$ and $H=(V',E')$, and if $b$ and $t$ are the bottom and top, respectively, of $H$; then we define $G\oslash_n H= (V'',E'')$ by
\[V''=V\cup (E_n\times (V'\setminus \{b,t\}))\]
and
\begin{align*}
E''=\{e&\in E\ |\ e\cap (V_n\cup V_{n+1})=\emptyset\}\\ 
&\cup \{\{u,(e,v)\}\ |\ e\in E_n,\ u\in V_n\cap e, \mbox{ and } \{b,v\}\in E'\}\\
&\cup \{\{u,(e,v)\}\ |\ e\in E_n,\ u\in V_{n+1}\cap e, \mbox{ and } \{v,t\}\in E'\}\\
& \cup \{\{(e,u),(e,v)\}\ |\ e\in E_n \mbox{ and } (\{u,v\}\setminus\{b,t\})\in E'\}.\\
\end{align*}

The formal definition of $G\oslash H$ is similar (just remove the subscripts and the first term in definition of $E''$).  It is clear that $G\oslash H$ can be created by performing $\oslash_n$-products repeatedly until all the edges that were originally in $G$ have been replaced.  

\begin{lemma}
\label{oslashn}
Given $n\in [M]$, the graph $T_{W,\kappa}\oslash_n T_{W',\kappa}$ is graph-isomorphic to $T_{W'',\kappa}$, where $W''=(w_r'')_{r=0}^{M+M'+1}\subseteq \mathbb{N}_0$ is defined by
\[w_r''=\begin{cases}
w_r & 0\leq i\leq n\\
\max\{w_n,w_{n+1}\}+w_{i-n}' & n< r< n+M'+1\\
w_{r-M'} & n+M'+1\leq  r \leq M+M'+1.
\end{cases}\]
\end{lemma}

\begin{proof}
We simply provide the graph isomorphism, and leave the details to the reader.  Define $F\colon T_{W,\kappa}\oslash T_{W',\kappa}\to T_{W''}$ by
\[
F(v)=\begin{cases}
(r,A) & v=(r,A) \mbox{ and } r\leq n\\
(n+r, B^\frown C) & 
 v=(\{(n,A),(n+1,B)\},(r,C))
\mbox{ and } A\preceq B
\\
(n+r,A^\frown C) &  v=(\{(n,A),(n+1,B)\},(r,C)) \mbox{ and } B\preceq A\\
(r+M'+1,A) & v=(r,A) \mbox{ and } r>n
\end{cases}\]
for each $v\in V(T_{W,\kappa}\oslash_n T_{W',\kappa})$.
\end{proof}

With repeated application of Lemma \ref{oslashn}, we obtain the following formula.

\begin{proposition}
\label{oslash}
The graph $T_{W,\kappa}\oslash T_{W',\kappa}$ is graph-isomorphic to $T_{W'',\kappa}$ where\newline $W''=(w_r'')_{r=0}^{(M+1)(M'+1)}\subseteq \mathbb{N}_0$ is defined by $w_0''=0$ and   
\[w_r''=\begin{cases}
\max\{w_n,w_{n+1}\}+w_{r-n(M'+1)}' & n(M'+1)< r< (n+1)(M'+1)\\
w_{n+1} & r = (n+1)(M'+1)
\end{cases}\]
for all $n\in[M]$.
\end{proposition}

We now fix $W''$ obtained in Proposition \ref{oslash}.  For what follows, we define the functions $x'$ and $y'$ for $W'$, and $x''$ and $y''$ for $W''$, in the same way $x$ and $y$ were defined for $W$ at the end of Section \ref{sec:notation}.

\begin{corollary}
\label{x''}
For each $n\in [M]$, let $K_n=\max\{w_n,w_{n+1}\}$. Then for all $n\in[M]$, $r\in [(n+1)(M'+1)]\setminus [n(M'+1)]$, and $i\in \mathbb{N}_0$,
\[x''(r,i)=\begin{cases}
(M'+1)x(n+1,i) & r=(n+1)(M'+1)\\
(M'+1)x(n,i) & r\neq (n+1)(M'+1) \mbox{ and }i\leq K_n\\
 n(M'+1) +x'(r-n(M'+1),i-K_n) & \mbox{otherwise},
\end{cases}\]
\[y''(r,i)=\begin{cases}
(M'+1)y(n+1,i) & r=(n+1)(M'+1) \mbox{ or } i\leq K_n\\
 n(M'+1) +y'(r-n(M'+1),i-K_n) & \mbox{otherwise}.
\end{cases}\]
\end{corollary}

In the next lemma we define $p_{W'}$ for $W'$ and $p_{W''}$ for $W''$ in the same way $p_W$ was defined for $W$ before Lemma \ref{lower_bound} in the last section.

\begin{lemma}
\label{cor:pw}
The parameter $p_{W''}$ satisfies the inequality $p_{W''}\leq \max\{p_W,p_{W'}\}$.
\end{lemma}

\begin{proof}
Let $r\in [(M+1)(M'+1)]$ and $i\in \mathbb{N}_0$ be such that $r\geq (x''(r,i)+y''(r,i))/2$ and suppose first that $r=(n+1)(M'+1)$ for some $n\in [M]$ (the case $r=0$ is trivial).
Then after using Corollary \ref{x''} and dividing by $M'+1$, we see that $n+1\geq (x(n+1,i)+y(n+1,i))/2$, which, by definition of $p_W$, implies $x(n+1,i+p_W)\geq (x(n+1),i)+y(n+1,i))/2$.
By multiplying by $M'+1$ and again using Corollary \ref{x''}, we see that $x''(r,i+p_W)\geq (x''(r,i)+y''(r,i))/2$.

Suppose now that $n(M'+1)<r<(n+1)(M'+1)$ for some $n\in[M]$ and $i> K_n$ (where $K_n=\max\{w_n, w_{n+1}\})$.
Then after using Corollary \ref{x''} and subtracting $n(M'+1)$, we see that $r-n(M'+1)\geq (x'(r-n(M'+1),i-K_n)+y'(r-n(M'+1),i-K_n))/2$, which, by definition of $p_{W'}$ implies $x'(r-n(M'+1),i-K_n+p_{W'})\geq  (x'(r-n(M'+1),i-K_n)+y'(r-n(M'+1),i-K_n))/2$.
By adding $n(M'+1)$ and again using Corollary \ref{x''}, we see that $x''(r,i+p_{W'})\geq (x''(r,i)+y''(r,i))/2$.

Suppose finally that $n(M'+1)<r<(n+1)(M'+1)$ for some $n\in[M]$ and $i\leq K_n$.  
Suppose further that $w_n<i\leq w_{n+1}$.
Then $x(n,i)=n$ and $y(n+1,i)>n+1$.
Thus, after applying Corollary \ref{x''} and then dividing by $M'+1$, we have 
\[n+1>r/(M'+1)\geq (n+y(n+1,i))/2,\] which implies $n+1\geq y(n+1,i)$, a contradiction.
Therefore, in order for the hypothesis to hold true, either $i\leq \min\{w_n,w_{n+1}\}$ or $w_{n+1}<w_n$.  In either case, $y(n,i)=y(n+1,i)$.
So after applying Corollary 6.3 and dividing by $M'+1$, we have 
\[n+1> (x(n,i)+y(n+1,i))/2=(x(n,i)+y(n,i))/2.\]  This means $n\geq (x(n,i)+y(n,i))/2$ and so, by definition of $p_W$, $x(n,i+p_W)\geq (x(n,i)+y(n,i))/2=(x(n,i)+y(n+1,i))/2$.  After multiplying by $M'+1$ and again applying Corollary \ref{x''}, we have 
\begin{align*}
x''(r,i+p_W)&\geq x''(n(M'+1),i+p_W)\\
&\geq (x''(n(M'+1),i)+y''((n+1)(M'+1),i))/2\\
&=(x''(r,i)+y''(r,i))/2.
\end{align*}
We have shown that for all $r\in[(M+1)(M'+1)]$ and $i\in\mathbb{N}_0$, $x''(r,i+\max\{p_W,p_{W'}\})\geq (x''(r,i)+y''(r,i))/2$ whenever $r\geq (x''(r,i)+y''(r,i))/2$.  Similarly it can be shown $y''(r,i+\max\{p_W,p_{W'}\})\leq (x''(r,i)+y''(r,i))/2$ whenever $r\leq (x''(r,i)+y''(r,i))/2$.  Therefore $p_{W''}\leq \max\{p_W,p_{W'}\}$.
\end{proof}

\begin{definition}
\label{def:family}
Given a bundle graph $G$, the \emph{family of bundle graphs generated by $G$} is the set $\left\{G^{\oslash^k}\right\}_{k=1}^\infty$, where $G^{\oslash^k}$ is defined recursively by $G^{\oslash^1}=G$ and $G^{\oslash^{k+1}}=G^{\oslash^k}\oslash G$ for all $k\in\mathbb{N}$.
\end{definition}

There are a few families of bundle graphs which have earned special names.  Given a cardinality $\kappa$, the family of $\kappa$-branching diamonds is the family generated by $T_{(0,1,0),\kappa}$, the family of $\kappa$-branching Laakso graphs is the family generated by $T_{(0,0,1,0,0),\kappa}$, and the family of $\kappa$-branching parasol graphs is the family generated by $T_{(0,0,1,0),\kappa}$.  Lemma \ref{cor:pw} and Theorem \ref{thm:esa} yield the following corollary.

\begin{corollary}
\label{cor:distortion}
Given a finitely branching bundle graph $G$, the family of bundle graphs generated by $G$ is equi-bi-Lipschitzly embeddable into any Banach space with an ESA basis, with distortion bounded above by a constant not depending on the target space or branching number of $G$.  In particular, every finitely branching diamond, Laakso, and parasol graph is bi-Lipschitzly embeddable into any Banach space with an ESA basis with distortion bounded above by $8$.
\end{corollary}

Now that we've formally defined families of bundle graphs generated by a base graph, we come to the characterizations of Banach space properties via non-equi-bi-Lipschitz embeddability of families of graphs, as promised in the introduction.  Work done by Brunel and Sucheston (\cite{brunel_sucheston} and \cite{brunel_sucheston2}, see also Theorem 2.3 in \cite{ostrovskii_randrianantoanina}), shows that for every non-reflexive Banach space $X$, there is a Banach space with an ESA basis which is finitely representable in $X$; and so every family of bundle graphs generated by a finitely-branching bundle graph is equi-bi-Lipschitzly embeddable into any non-reflexive Banach space by Theorem \ref{thm:esa}.  Conversely, a consequence of Lemma 1 in Section 4 of \cite{johnson_schechtman}, says that the family of binary (that is, 2-branching) diamond graphs is not equi-bi-Lipschitzly embeddable into any Banach space with uniformly convex norm.  Virtually the same proof shows that, in fact, every family of bundle graphs generated by a nontrivial bundle graph is not equi-bi-Lipschitzly embeddable into a Banach space with uniformly convex norm.  Every superreflexive Banach space is uniformly convexifiable, so we have the following characterization(s) of superreflexivity.

\begin{theorem}
\label{thm:superreflexivity}
Fix a nontrivial finitely branching bundle graph $G$.  Then a Banach space $X$ is superreflexive if and only if the family of bundle graphs generated by $G$ is non-equi-bi-Lipschitzly embeddable into $X$.
\end{theorem}

\begin{remark}
Johnson and Schechtman \cite{johnson_schechtman} obtained a distortion bound of $16+\varepsilon$ for the equi-bi-Lipschitz embeddability of the family of binary diamond graphs into a non-superreflexive Banach space.  Ostrovskii and Randrianantoanina \cite{ostrovskii_randrianantoanina} improved and generalized this, obtaining a distortion bound of $8+\varepsilon$ for any family of finitely-branching diamond or Laakso graphs.  Corollary \ref{cor:distortion} yields a further generalization, and recovers the same distortion bound of $8+\epsilon$ for any family of finitely-branching diamond, Laakso, or parasol graph.
\end{remark}

Theorem 3.2 in \cite{baudier_etal} shows that within the class of reflexive Banach spaces with an unconditional structure, a Banach space which is not asymptotically uniformly convexifiable will contain $(1+\varepsilon, 1+\varepsilon)$-good $\ell_\infty$ trees of arbitrary height for all $\varepsilon>0$.  Theorem 4.1 in \cite{baudier_etal} then shows that every family of bundle graphs generated by a nontrivial infinitely-branching bundle graph is not equi-bi-Lipschitzly embeddable into any Banach space which is asymptotically midpoint uniformly convexifiable.  Thus we have the following metric characterization(s) of asymptotic uniform convexifiability within the class of reflexive Banach spaces with an unconditional asymptotic structure.

\begin{theorem}
\label{thm:auc_characterization}
Fix a nontrivial $\aleph_0$-branching bundle graph $G$.  Then a reflexive Banach space $X$ with an unconditional asymptotic structure is asymptotically uniformly convexifiable if and only if the family of bundle graphs generated by $G$ is non-equi-bi-Lipschitzly embeddable into $X$.
\end{theorem}

\textbf{Acknowledgments:} The author was partly supported by NSF grants DMS-1464713 and DMS-1565826.  He would like to thank his advisors Thomas Schlumprecht and Florent Baudier for their endless support and Bill Johnson for being a constant and invaluable resource.


\begin{thebibliography}{10}

\bibitem{argyros_motakis_sari}
S. Argyros, P. Motakis, and B. Sari,  \emph{A study of conditional spreading sequences}, J Funct. Anal. 272 (2017), no. 3, 1205--1257.

\bibitem{baudier}
F. Baudier, \emph{Metrical characterization of super-reflexivity and linear type of Banach spaces}, Arch. Math. 89 (2007), 419--429.

\bibitem{baudier_etal}
F. Baudier, R. Causey, S. Dilworth, D. Kutzarova, N. L. Randrianarivony, Th. Schlumprecht, and S. Zhang, \emph{On the geometry of the countably branching diamond graphs}, J. Funct. Anal. 273 (2017), no. 10, 3150--3199.

\bibitem{bourgain}
J. Bourgain, \emph{The metrical interpretation of superreflexivity in Banach spaces}, Israel J. Math. 56 (1986), 222---230.

\bibitem{brunel_sucheston}
A. Brunel and L. Sucheston, \emph{Equal signs additive sequences in Banach spaces},  J. Funct. Anal. 21 (1976), no. 3, 286--304.

\bibitem{brunel_sucheston2}
A. Brunel and L. Sucheston, \emph{On J-convexity and some ergodic super-properties of Banach spaces}, Trans. Amer. Math. Soc. 204 (1975), 79-–90.

\bibitem{johnson_schechtman}
W. B. Johnson and G. Schechtman, \emph{Diamond graphs and super-reflexivity}, J. Anal. 1 (2009), no. 2, 177--189.

\bibitem{naor}
A. Naor, \emph{An introduction to the Ribe progam}, Jpn. J. Math. 7 (2012), no. 2, 167--233.

\bibitem{ostrovskii_randrianantoanina}
M. I. Ostrovskii and B. Randrianantoanina, \emph{A new approach to low-distortion embeddings of finite meric spaces into non-superrefleive Banach spaces}, J. Funct. Anal. 273 (2017), no. 2, 598--651.

\end{thebibliography}
\end{document}